\documentclass{amsart}
\usepackage{amsmath}
\usepackage{amsfonts}
\usepackage{amssymb}
\usepackage{amscd}
\usepackage[all]{xy}
\usepackage[dvipsnames]{xcolor}
\usepackage{tikz-cd} 
\usepackage{pdfsync}
\usepackage{url}
\usepackage{algorithm}
\usepackage{algorithmic}

\usepackage{framed} % For leftbar

\usepackage{amssymb,amsmath,amsthm,graphicx,amscd,eufrak,hyperref,}
\theoremstyle{definition}

\newtheorem{thm}{Theorem}[section]
\newtheorem*{mainresult}{Theorem}

\newtheorem{cor}[thm]{Corollary}

\newtheorem{condition}[thm]{Condition}

\theoremstyle{definition}
\newtheorem{remark}[thm]{Remark}

\newtheorem{example}[thm]{Example}

\newcommand{\CC}{{\mathbb C}}

\def\Jac{{\rm Jac}\,}

\newcommand{\Conormal}{\operatorname{Con}}
\newcommand{\DL}{\operatorname{DL}}

\DeclareMathOperator{\Graph}{Graph}
\DeclareMathOperator{\SO}{SO}
\DeclareMathOperator{\Sym}{Sym}
\DeclareMathOperator{\reg}{reg}

\theoremstyle{definition}

\newcommand{\defcolor}[1]{{\color{RoyalBlue}#1}}
\newcommand{\demph}[1]{\defcolor{{\sl #1}}} % emph colored
\newcommand{\JIR}[1]{{\color{black}#1}} % Jose's comments
\newcommand{\EH}[1]{{\color{black}#1}} % Emil's comments

\begin{document}
	\title[Data loci in algebraic optimization]{
		\bf Data loci in algebraic optimization}
	\author{Emil Horobe\c{t} and Jose Israel Rodriguez}
	
	
	\maketitle
	
	\begin{abstract}
		We consider parametric optimization problems from an algebraic viewpoint. 
		The idea is to find all of the critical points of an objective function thereby determining a global optimum. 
		For generic parameters (data) \JIR{in the objective function} the number of critical points remains constant. This number is known as the algebraic degree of an optimization problem.	
		In this article, we go further by considering the inverse problem of finding parameters \JIR{of the objective function so it gives}  rise to critical points exhibiting a special structure. \JIR{For example  if the critical point is in the singular locus, has some symmetry, 
	or satisfies some other algebraic property.}
		Our main result is a  theorem describing such parameters. 
		% conditions when among the resulting local minima and maxima there is at least one which satisfies a given polynomial conditions.
%		In this article we provide examples, methods and algorithms to determine conditions on the parameters of some parametric optimization problems. These conditions determine when among the resulting local minima and maxima there is at least one which satisfies a given polynomial conditions (for example it is singular or symmetric). 
	\end{abstract}

	\section{Introduction}
	In many kinds of optimization problems (distance optimization, optimizing communication rate etc.) it is interesting to ask the question \JIR{of} whether the solution is satisfying certain meaningful (polynomial) conditions. For example one can be interested if the solution will be singular, have symmetry, 
	have coordinates summing to one, etc.
	Equally interesting is to ask the same question not only about the minimizer of the optimization problem, but about all of the local minima and maxima as well. 
	Even stronger, we want conditions for all of the critical points of the problem.
	In other words, we are considering a generalization of the inverse problem of determining the parameter data that exhibits a certain type of critical point of the objective function. 
	We provide examples, methods and algorithms to test \JIR{the} properties of a parametric optimization problem. 
	\JIR{This a well studied topic and a part of parametric polynomial system solving about which one can read in \cite{DR07, Lass}.}
	
	Our motivating example is the \demph{scaled distance function}. 
	A scaled distance function on $\mathbb{R}^n$ is prescribed by \JIR{a} parameter data $u\in\mathbb{R}^n$ and a fixed scaling vector $w\in\mathbb{R}^n$ as 
	\[
	d_{u}^{w}: \mathbb{R}^n\to \mathbb{R},
	\quad
	x\mapsto \displaystyle\sum_{i=1}^n w_i(u_i-x_i)^2.	
	\] 
	\JIR{For this special case, our} main problem is to solve 
	\begin{equation}\label{MinProblem1}
	\begin{cases}
	\text{minimize } d_{u}^{w}(x) \\
	\text{subject to }x\in X_{\mathbb{R}},
	\end{cases}
	\end{equation}
	where $X_{\mathbb{R}}$ is an affine variety in $\mathbb{R}^n$. We want to provide the set of parameters $u\in\mathbb{R}^n$ for which at least one of the critical points of Problem~\ref{MinProblem1} satisfy prescribed (polynomial) conditions.
	
	To use the algebraic techniques, we pose our problem over the algebraically closed field of complex numbers:
	For an \demph{objective function} $d_{u}^{w}:\mathbb{C}^n\to\mathbb{C}$
	given by 
	%$\displaystyle\sum_{i=1}^n w_i(u_i-x_i)^2$ 
	$u\in \mathbb{C}^n$ and $w\in \mathbb{C}^n$, determine the critical points of $d_{u}^{w}|_X$ where $X$ denotes Zariski closure of $X_{\mathbb{R}}$.
	A practically meaningful (e.g. positive, real) critical point of $d_{u}^{w}|_{X_\mathbb{R}}$ will be among the set of complex critical  points of $d_{u}^{w}|_{X}$.
	
	\JIR{A special case of our main theorem (Theorem~\ref{Structurethm}) is stated next.}
	\begin{mainresult}[Special Case: Distance Function]\label{thm_beginning}
		% Need a domain and range for $\Gamma$.
		Let $A$ be a subvariety of $X$, then the variety of parameters $u\in \mathbb{C}^n$ for which Problem~\ref{MinProblem1} will have a critical point in $A$ is given by the closure of
		\[
		\Gamma(\Conormal(X)\cap\Conormal(A))\subset\mathbb{C}^n\times\mathbb{C}^{n},
		\]
		where $\Conormal(X)$ and $\Conormal(A)$ are the corresponding (affine) conormal varieties (see \ref{conormal}) in $\mathbb{C}^n\times\mathbb{C}^n$ 
		and \[\Gamma: \Conormal(X)\to \mathbb{C}^n,\quad  \Gamma(x,y)=\left(x_1 + \frac{1}{w_1}y_1,\dots, x_n +\frac{1}{w_n}y_n \right).\]
	\end{mainresult}
		
	\begin{remark}
		The conormal varieties  $\Conormal(X)$ and $\Conormal(A)$ 
		are $n$-dimensional in $\mathbb{C}^n\times\mathbb{C}^n$.
		\JIR{See  Section~\ref{ss:conormal} for definitions.}
		%One would expect by a version of Bertini's theorem~\ref{}, 
		So one would expect the intersection 
		$\Conormal(X)\cap \Conormal(A)$
		to be a finite set of points. 
		But when $A\subset X$, the intersection is much more interesting and can be positive dimensional.  
		For example if $A$ is a regular point in $X$, then the intersection 
		$\Conormal(X)\cap \Conormal(A)$ is the normal space at the regular point.
		%is expected to be $0$-dimensional right? This is because the conormal variety is $n$-dimensional 
		%and also codimension $n$. 
		%But it doesn't have to be zero dimensional as some of our examples will probably show. 
	\end{remark}

	\JIR{Our main result Theorem~\ref{Structurethm} } provides a framework for generalizing the previous theorem to 
	other optimizations problems besides the scaled distance function. 

\JIR{
\begin{mainresult}[Main theorem]
	Let $A\subseteq X$ be a subvariety and suppose that optimization problem \eqref{MinProblem} satisfies Condition~\ref{Condition_on_d}. Then, we have that
	\[
	\DL_A=\overline{\Gamma(\Conormal(X)\cap\Conormal(A)\setminus\mathcal{H})}\subset\mathbb{C}^n,
	\]
where $\mathcal{H}$  and $\Gamma$ are as in Condition~\ref{Condition_on_d}.
\end{mainresult}
}

\JIR{Our article is organized as follows.}	
\JIR{In Section~\ref{sec:conormal} we introduce notation and conditions for our main result.
In Section~\ref{Sec_Data_loci}  we prove the main result and give some corollaries regarding containments of the data locus. }
In Section~\ref{Sec_Exa} we discuss classical distance optimization concerning low rank and structured low rank approximations of matrices and tensors in Examples~\ref{rank_var}, \ref{form_control}, \ref{Hankel_exa}, \ref{low_rank}; we cover weighted distance optimization in Example~\ref{rat_norm}; we relate to the maximum likelihood (ML) degree in Remark~\ref{MLDSL}; \JIR{finally,} we show in Example~\ref{water} that optimizing communication rate also fits our problem setting.
	
	\section{Conormal map derived from the objective function}\label{sec:conormal}
	
	\subsection{ Pairing data with $\Gamma$}
	We are solving parametric optimization problems with polynomial constraints. More precisely we have the following parametric optimization problem: for any fixed vector of parameters $u\in \mathbb{R}^n$ solve 
	
	\begin{equation}\label{MinProblem}
	\begin{cases}
	\text{minimize } d_{u}(x), \\
	\text{subject to }x\in X_\mathbb{R},
	\end{cases}
	\end{equation}
	where $X$ is an affine variety in $\mathbb{R}^n$ that is the common zero set of the polynomials $f_1(x),\ldots,f_c(x)$ and $d_u(x)$ is some parametric objective function. 
	We will denote the collection of regular points by $X_{\reg}$.
	
	In order to find the minimizer algebraically we have to find all local minimizers and maximizers for \eqref{MinProblem}, hence all the \demph{constrained critical points} of $d_u$.  
	The set of constrained critical points of $d_{u}$ is  given by following:
	\begin{equation}\label{CritEqu}
	\big\{
	x\in\mathbb{C}^n 
	\,|\,
	\nabla d_u(x) \in \mathcal{N}_x X,
	\,
	x\in X_{\reg}
	\big\}.
	\end{equation}
	Here $\nabla d_u(x)$ denotes the partial derivatives of $d_u$ with respect to $x$, 
	and $\mathcal{N}_x X$ is the \demph{normal space} of $X$ at $x$, 
	\JIR{that is for a regular point $x\in X$
	\[
	\mathcal{N}_x X
		:=
		\left\{ y\in \CC^n : \text{ rank of } 
		\begin{bmatrix}
		y\\
		\nabla f_1(x)\\
		\vdots\\
		\nabla f_c(x)
		\end{bmatrix}
		\text{ is equal to the codimension of } X
		\right\}.
	\]
	}	
	In other words, a point $x\in X_{\reg}$ is a solution to \eqref{CritEqu} if and only if $\nabla d_u(x)=y$, for some $y\in \mathcal{N}_x X$.  
	
	\begin{remark}
		A classical approach to such problems is to use \demph{Lagrange multipliers}.
		If $X$ is defined by $f_1,\dots,f_k$, 
		 when $X$ is a complete intersection of $c$ hypersurfaces, see \cite[Section 5.5.3]{ConvOpt}. Indeed in this case, a solution $x$ in  \eqref{CritEqu} 
		 corresponds to a solution $(x,\lambda_1,\dots,\lambda_c)$ to
%
		 %\JIR{the solutions $x\in \mathbb{C}^n$ of this system}		
	\begin{equation}\label{LagrangeCritEqu}
%		(x,\lambda)\text{ is a solution to }		
		\begin{cases}
		\displaystyle\nabla d_u(x) + \sum_{i=1}^c \lambda_i \nabla f_i(x) = 0\\
		f_1(x) = f_2(x) = \ldots  = f_c(x) = 0.
		\end{cases}
		\end{equation}
	\end{remark}
\JIR{Moreover, the reverse correspondence holds. In the complete intersection case, the correspondence is one to one, but in the non-complete intersection there are infinitely many choices for $\lambda_1,\dots,\lambda_k$.}

\JIR{If one doesn't have a complete intersection, then choose $f_1,\dots,f_s$ a generating set of the radical ideal of $X$ and then solutions of \eqref{CritEqu} satisfy  $f_i(x)=0$ for $i=1,\dots, s$ and 
\[
		\left\{ x\in \CC^n : \text{ rank of } 
		\begin{bmatrix}
		\nabla d_u(x)\\
		\nabla f_1(x)\\
		\vdots\\
		\nabla f_s(x)
		\end{bmatrix}
		\text{ is less than or equal to the codimension of } X
		\right\}.
\]
This is equivalent to \eqref{LagrangeCritEqu} in the case $X$ is a complete intersection.
}

	Furthermore, to be able to use algebraic geometry techniques, we will need some assumptions on $d_u$. 
	\begin{condition}\label{Condition_on_d}
	Suppose there exists a hypersurface $\mathcal{H}\subset\mathbb{C}^n\times\mathbb{C}^n$ such that for all $(x,y)\in \mathbb{C}^n\times\mathbb{C}^n\setminus\mathcal{H}$,  
	there is a unique $u\in\mathbb{C}^n$, such that \[\nabla d_u(x)=y.\] 
		Now let us denote by \demph{$\Gamma(x,y):=u$ the unique solution} $u$ to $\nabla d_u(x)=y$, for any fixed $x$ and $y$. Furthermore we must require that $\Gamma$ is a rational function.
	\end{condition}
\JIR{
One way to view this condition, is that for generic $(x,y)\in \CC^n\times \CC^n$ there is a unique solution $u\in \CC^n$ to $\nabla d_u(x)=y$.
}

\JIR{	
\begin{remark}
One example where this condition doesn't hold is for the $p$-norm. 
However, there are techniques \cite{KKS21} to handle this case in a similar fashion to what we present here.
\end{remark}
	}
	
	In this setting, a point $x\in X_{\reg}$ is a solution to \eqref{CritEqu} if and only if there exists $y\in \mathcal{N}_x X$, such that \[\Gamma(x,y)=u.\] 
	A couple of examples for $\Gamma$ from classical optimization problems will follow.
	
	\begin{example}[Weighted distance optimization]\label{dist_opt_exam}
		
		Suppose that we want to solve for any parameter vector $u\in\mathbb{R}^n$ and fixed weight vector \JIR{$w\in(\mathbb{R}\setminus \{0\})^n$}, the weighted distance minimization problem 
		\[
		\begin{cases}
		\text{minimize} \displaystyle\sum_{i=1}^n w_i(u_i-x_i)^2,\\
		\text{subject to }x\in X_{\reg},
		\end{cases}
		\]
		with $X$ an affine variety in $\mathbb{R}^n$. Then, based on the discussion above, we get that all the critical points of the system (local minimizers and maximizers) satisfy that
		\[
		\begin{cases}
		\nabla \displaystyle\sum_{i=1}^n w_i(u_i-x_i)^2\in \mathcal{N}_x X,\\
		x\in X_{\reg}.
		\end{cases}
		\]
		That is equivalent to
		\[
		\begin{cases}
		(w_1(u_1-x_1),\ldots,w_n(u_n-x_n))\in \mathcal{N}_x X,\\
		x\in X_{\reg}.
		\end{cases}
		\]
		So we get that $x\in X_{\reg}$ is a critical point of the system if and only if there exists $y\in \mathcal{N}_x X$ such that $(w_1(u_1-x_1),\ldots,w_n(u_n-x_n))=y$. 
		\JIR{We see that in this case, Condition~\ref{Condition_on_d} is satisfied  for all $(x,y)\in\mathbb{C}^n\times \mathbb{C}^n$. 
		So for any fixed pair $(x,y)$, there is a unique $u$} such that the above condition is satisfied, namely \[\displaystyle u=\Gamma(x,y)=\left(x_1+\frac{1}{w_1}y_1,\ldots, x_n+\frac{1}{w_n}y_n\right).\]
		%So any critical point  $x\in X_{\reg}$ of the problem satisfies that there exists $y\in \mathcal{N}_xX$ such that
		%\[
		%u=\Gamma(x,y),
		%\] with $\Gamma(x,y)=\left(x_1+\frac{1}{w_1}y_1,\ldots, x_n+\frac{1}{w_n}y_n\right)$.
		Observe that when we choose the weight function to be all ones then we get back to classical distance optimization on algebraic varieties, namely to Euclidean Distance Degree theory, which is discussed in \JIR{the} article \cite{DHOST16} in detail.
		\JIR{Recently, an approach for understanding critical points of weighted distance functions on a variety as limits of critical points of perturbed functions has been explored in \cite{MR4048615} from the viewpoint of algebraic topology and addressed the multiview conjecture in~\cite{DHOST16}.}
		
		%\begin{remark}
		%Observe that not all weighted distance functions satisfy Condition~\ref{Condition_on_d}. If any of the weights is equal to zero then for fixed $x$ and $y$ there might not be any $u$ or there might be infinitely many $u$'s, such that $\nabla d_u(x)=y$.
		%\end{remark}
	\end{example}
	
	\begin{example}[Information theory-Water filling]\label{water}
		
		We consider the following paramet\-rized optimization problem in convex optimization (see \cite[Chapter $5$, Example~$5.2$]{ConvOpt}) \EH{that is also explored in~\cite{WaterFilling}.} For any parameter vector $u\in \mathbb{R}^n,$ with $u_i>0$
		\[
		\begin{cases}
		\text{minimize } \displaystyle\sum_{i=1}^n -\log(u_i+x_i),\\
		\text{subject to }x_i\geq 0\text{ and }\displaystyle\sum_{i=1}^n x_i=1.
		\end{cases}
		\]
		This problem arises in information theory, in allocating power to a
		set of $n$ communication channels. The variable $x_i$ represents the transmitter power allocated to the $i$-th channel, and $\log(u_i + x_i)$ gives the capacity or communication rate of the channel, so the problem is to allocate a total power of one to the channels, in order to maximize the total communication rate.
		A relaxed version over the complex numbers  of this problem can be tackled by the methods presented in this article. So we consider the following optimization problem. For any parameter vector $u\in \mathbb{C}^n,$ 
		\[
		\begin{cases}
		\text{minimize } \displaystyle\sum_{i=1}^n -\log(u_i+x_i),\\
		\text{subject to }x\in \mathcal{L},
		\end{cases}
		\] where $\mathcal{L}$ is a suitable linear space and $\log$ is some fixed branch of the complex logarithm. We get that all the critical points of the system (local minimizers and maximizers) satisfy that
		\[
		\begin{cases}
		\nabla \displaystyle\sum_{i=1}^n -\log(u_i+x_i)\in \mathcal{N}_x \mathcal{L},\\
		x\in \mathcal{L}.
		\end{cases}
		\]
		That is equivalent to
		\[
		\begin{cases}
		\left(\frac{1}{u_1+x_1},\ldots,\frac{1}{u_n+x_n}\right)\in \mathcal{N}_x \mathcal{L},\\
		x\in \mathcal{L}.
		\end{cases}
		\]
		%So we get that $x\in \mathcal{L}$ is an critical point of the system if and only if there exists $y\in \mathcal{N}_x \mathcal{L}$, such that $\left(\frac{1}{u_1+x_1},\ldots,\frac{1}{u_n+x_n}\right)=y$.
		We see that in this case Condition~\ref{Condition_on_d} is satisfied, so for $(x,y)\in\mathbb{C}^n\times \mathbb{C}^n\setminus \mathcal{H}$, where $\mathcal{H}$ is defined by $y_1\cdot\ldots\cdot y_n=0$, there is a unique $u$, such that the above condition is satisfied, namely
		\[\displaystyle u=\Gamma(x,y)=\left(\frac{1}{y_1}-x_1,\ldots, \frac{1}{y_n}-x_n\right).\]
		
	\end{example}
	In the next section we introduce the conormal variety and construct the graph of $\Gamma$ over it.
	
	\subsection{Conormal Variety and $\Gamma$-correspondence}\label{ss:conormal}
	
	\JIR{We begin this section by recalling cornormal varieties and in the case of affine cones and projective varieties,
we refer the reader to \cite[Section 5]{DHOST16}, \cite[Chapter 1]{GKZ} for a complete discussion.
For general affine varieties we continue the discussion.
}
	We have seen so far that pairs of points $(x,y)$ such that $x\in X_{\reg}$ and $y\in \mathcal{N}_xX$ play a crucial role in our analysis. The closure of the collection of all such pairs with $x\in X_{\reg}$  \JIR{we call} the (affine) \demph{conormal variety} and we denote it by $\Conormal(X)$.
	\JIR{In an analogous way to \cite[Chapter 1]{GKZ} the affine cornormal for an irreducible variety is defined \cite[Def 4.6.1.7]{Land} as}	
	\begin{equation}\label{conormal}
	\Conormal(X)=\overline{\{(x,y)\in \mathbb{C}^n_x\times \mathbb{C}^n_y,\ \ x\in X_{\reg},\ y\in\mathcal{N}_xX\}}.
	\end{equation}
	We use the notation $\mathbb{C}^n_x\times \mathbb{C}^n_y$ instead of just simply $\mathbb{C}^n\times \mathbb{C}^n$ to keep track that the first $n$ tuple of coordinates represents a point $x$ and the second $n$ tuple of coordinates represents a point $y$.
	
	There is a natural pair of projections $\pi_1:\Conormal(X)\to \mathbb{C}^n_x$ to the first $n$-tuple of coordinates and $\pi_2:\Conormal(X)\to \mathbb{C}^n_y$ to the second $n$-tuple of coordinates. The image of the first projection is the variety $X$ itself and the closure of the image of the second projection $Y:=\overline{\pi_2(\Conormal(X))}$.
	 \JIR{In the case of affine cones (or projective varieties), is called the {dual variety} of $X$, see for instance~\cite[Section 5.4.2]{RS13} and \cite[Chapter 1, Section 1]{GKZ}}
	
	\[\begin{tikzcd}
	&\Conormal(X)\arrow{ld}[swap]{\pi_1}\arrow{rd}{\pi_2} \\
	X\subseteq \mathbb{C}^n_x & & Y\subseteq \mathbb{C}^n_y
	\end{tikzcd}
	\]
\JIR{
	Similar to \cite[Theorem 4.1]{DHOST16}) the first projection restricted to $X_{\reg}$ is an affine vector bundle of rank equal to the codimension of $X$.
	Since the sum of the codimension and dimension equals $n$ it results that dimension of the conormal inside $\CC^n_x\times \CC^n_y$ is $n$. 
}	

\JIR{
When $X$ is a reducible variety, its conormal variety is defined to be the union of conormal varieties of each irreducible component. 
	}

\JIR{
\begin{remark}
When $X$ is an irreducible affine cone it makes sense to consider its cormal variety inside in $\mathbb{P}^{n-1}\times \mathbb{P}^{n-1}$:
\[
	\{ (\hat x, \hat y) \in \mathbb{P}^{n-1}\times \mathbb{P}^{n-1} : 
	\hat x \in \hat X_{\reg}  \text{ and } y\in \mathcal{N}_x X
	\} 
\]
This biprojective variety has codimension $n$ and dimension $n-2$. 
Projecting to the second factor $\mathbb{P}^{n-1}$ yields a subvariety. 
When this subvariety is a hypersurface, it is the dual variety of $X$. 
\end{remark}	
}
	
%    	%the conormal variety is an irreducible variety of dimension $n$ inside $\mathbb{C}_x^n \times \mathbb{C}^n_y$. 
%    	 The first projection is an affine vector bundle of rank $c$ over $X_{\reg}$ (where $c$ is the codimension of $X$). %Over generic $y_0 \in \mathbb{C}^n_y$, the second projection $\pi_2: \Conormal(X) \to \mathbb{C}^n_y$ has finite fibers $\pi_2^{-1}(y_0)$ of cardinality equal to the \demph{Euclidean Distance Degree}~\cite{DHOST16} of $X$.
	
	So we have that $x\in X_{\reg}$ is a solution to \eqref{CritEqu} if and only if there exists a point $(x,y)$ in $\Conormal(X)\setminus\mathcal{H}$ such that $u=\Gamma(x,y)$, equivalently such that $((x,y),u)$ is a point on the closure of the graph of $\Gamma$ over the conormal variety $\Conormal(X)\setminus\mathcal{H}$. Let us remember that $\Gamma: \Conormal(X)\setminus\mathcal{H}\subseteq \mathbb{C}_x^n\times \mathbb{C}_y^n \to \mathbb{C}_u^n$.
	We call the closure of the graph of $\Gamma$ over $\Conormal(X)\setminus\mathcal{H}$ \demph{ the $\Gamma$-correspondence}. 
	More formally $\Graph(\Gamma)$ is the closure of all triples
\begin{equation}\label{eq:gamm-graph}
	\{(x,y,u)\text{ such that, }\ (x,y)\in \Conormal(X)\setminus\mathcal{H} \text{ and } u=\Gamma(x,y)\}.
\end{equation}
	
	\begin{remark}
		In the case of classical distance optimization on varieties, like in Example~\ref{dist_opt_exam}, we have that $\Gamma(x,y)=x+y$ and the $\Gamma$-correspondence is called the \JIR{extended} ED (Euclidean Distance)-correspondence, 
		\JIR{while its projection onto $\CC^n_x\times\CC^n_u$ is the ED correspondence}, 
		see \cite[Section $4$]{DHOST16}.
	\end{remark} 
	We have the following diagram of projections from the $\Gamma$-correspondence.
	
	\[\begin{tikzcd}
	& &\Graph(\Gamma)\subseteq \mathbb{C}^n_x \times \mathbb{C}^n_y \times \mathbb{C}^n_u \arrow{ld}[swap]{\pi_{12}}\arrow{rd}{\pi_3}& \\
	&\Conormal(X)\subseteq \mathbb{C}^n_x \times \mathbb{C}^n_y\arrow{ld}[swap]{\pi_1}\arrow{rd}{\pi_2}& & \mathbb{C}_u^n \\
	X \subseteq \mathbb{C}^n_x& & Y\subseteq \mathbb{C}^n_y& 
	\end{tikzcd}
	\]
	Because $\Conormal(X)$ is an $n$-dimensional variety,
	$\Graph(\Gamma)$ is an $n$-dimensional variety as well inside $\mathbb{C}^n_x\times \mathbb{C}^n_y \times \mathbb{C}^n_u$. 
	
	\JIR{
	 If $\pi_3$ is a dominant map, then $\pi_3$ is a generically finite map. }
%	By Condition~\ref{Condition_on_d}, we have that $\pi_3$ is a generically finite map. 

	\section{Data loci}\label{Sec_Data_loci}
	
	In this article we want to determine the set of all parameters $u\in \mathbb{C}_u^n$, such that the optimization problem \eqref{MinProblem} has at least one critical solution in a given subvariety of $X$. 
	\JIR{
	Let $A\subseteq X$ be a subvariety of $X$ and suppose that optimization problem \eqref{MinProblem} \JIR{satisfies} Condition~\ref{Condition_on_d}.
	 In this case, we define the \demph{data locus of $A$} to be the closure of projection $\pi_3$ (that is into the space $\mathbb{C}_u^n$ of parameters $u$) of triples $(x,y,u)\in \Graph(\Gamma)$, such that $x\in A$.
	 }
	We denote the data locus by $\DL_A$ and formally we have that 
	\[
	\DL_A
	:=
	\overline{\pi_{3}
		\left(
		\Graph(\Gamma)\cap
		\left(A\times\mathbb{C}_{y}^{n}\times\mathbb{C}_{u}^{n}\right)
		\right)
	}.
	\]
	Our main theorem is \JIR{restated} next.
	\begin{thm}[Structure theorem]\label{Structurethm}
		Let $A\subseteq X$ be a subvariety and suppose that optimization problem \eqref{MinProblem} satisfies Condition~\ref{Condition_on_d}. Then, we have that
		\[
		\DL_A=\overline{\Gamma(\Conormal(X)\cap\Conormal(A)\setminus\mathcal{H})}\subset\mathbb{C}^n,
		\]
		where $\mathcal{H}$ is as in Condition~\ref{Condition_on_d}.
	\end{thm}
	A special case of this theorem, when $\Gamma(x,y)=(x_i+\frac{1}{w_i}y)_i$, \EH{is presented in the introduction.}
	\begin{proof}
	\JIR{The backbone of the proof is to show these two equalities:
%\begin{multline*}
\begin{equation}\label{eq:pi3}
	\pi_3 (\Graph(\Gamma) \cap ( A \times \CC^n_y \times \CC^n_u) ) 
	 \,\,=\,\,
	\Gamma(\Conormal(X)\setminus \mathcal{H})\, \cap  \, (A\times \CC^n_y)
\end{equation}
\begin{equation}\label{eq:conXA}
	(\Conormal(X)\setminus \mathcal{H} ) \cap (A\times \CC_y^n)
	\,\,=\,\,
	(\Conormal(X) \cap \Conormal(A)) \setminus \mathcal{H}.
\end{equation}
%\end{multline*}	
	}
	
\JIR{First to show \eqref{eq:pi3}, 
	we understand that}
		the graph of $\Gamma$ \eqref{eq:gamm-graph} over $\Conormal(X)\setminus\mathcal{H}$ is
		\[\{(x,y,u)\in \mathbb{C}_x^n\times \mathbb{C}_y^n \times \mathbb{C}_u^n,\ \text{ s.t. } (x,y)\in \Conormal(X)\setminus\mathcal{H} \text{ and } u=\Gamma(x,y)\}.\]
		So we have that $\Graph(\Gamma)\cap(A\times\mathbb{C}_y^n\times\mathbb{C}_u^n)$ equals to
		\[
		\{(x,y,u)\in \mathbb{C}_x^n\times \mathbb{C}_y^n \times \mathbb{C}_u^n,\ \text{ s.t. } x\in A, (x,y)\in \Conormal(X)\setminus\mathcal{H}\text{ and }u=\Gamma(x,y)\}.
		\]
\JIR{Applying $\pi_3$ to each side we get the equality in \eqref{eq:pi3}.}		
		\JIR{Now to show \eqref{eq:conXA}, note that the closure of }		
		 the set of pairs
\JIR{
		\[
		\{(x,y)\in \mathbb{C}_x^n\times\mathbb{C}_y^n,\text{ s.t. }x\in A_{\reg}\cap X_{\reg}\text{ and } (x,y)\in \Conormal(X)\setminus\mathcal{H}\}
		\]
	}
		is equal to \JIR{
		the closure of}
		\[
		\{(x,y)\in \mathbb{C}_x^n\times\mathbb{C}_y^n\setminus\mathcal{H},\text{ s.t. }x\in A_{\reg}\text{ and } y\in \mathcal{N}_xA \}\ \bigcap 
		\]
		\[
		\{(x,y)\in \mathbb{C}_x^n\times\mathbb{C}_y^n\setminus\mathcal{H},\text{ s.t. }x\in X_{\reg}\text{ and } y\in \mathcal{N}_xX\,.\]
\JIR{		In other words, this is equal to }
$\Conormal(X)\cap \Conormal(A)\setminus\mathcal{H}$,
		because $A\subseteq X$ is a subvariety for $x\in A_{\reg}\cap X_{\reg}$ 
we have that		
		 $\mathcal{N}_xX\subseteq \mathcal{N}_xA$.
		\JIR{By \eqref{eq:pi3} and \eqref{eq:conXA},}
		we get that 
		\[\Graph(\Gamma) \cap\left(A\times\mathbb{C}_{y}^{n}\times\mathbb{C}_{u}^{n}\right)\] is equal to
		the graph of $\Gamma$ over $\Conormal(X)\cap \Conormal(A)\setminus\mathcal{H}$.
\JIR{
For a pairs of point $(x,y)$ where $x\not \in A_{\reg}\cap X_{\reg}$ one can take a sequence of $x_k$ converging to $x$ such that $x_k$ are in $A_{\reg}\cap X_{\reg}$. Repeating the argument and taking limits yields the same final result. 
}
		
	\end{proof}
	Now that we know the structure of the data locus we can bound it as a set to be able to extract further properties of it.
	
	\begin{cor}[Bounding the data locus]\label{Bound}
		Let $A\subseteq X$ be a subvariety, then we have that
		\[
		\overline{\Gamma((A\times\{0\})\setminus\mathcal{H})}\subseteq \DL_A\subseteq \overline{\Gamma((A\times Y_A\cap Y)\setminus\mathcal{H})},
		\]
		where $Y_A$ is the closure of $\pi_2(\Conormal(A))$ and $Y$ is the closure of $\pi_2(\Conormal(X))$. 
		Moreover, if $X$ is contained in a hyperplane defined by 
		$h_0+h_1x_1+\cdots + h_n x_n=0$,  
		then 
		\[\overline{
			\Gamma\left(
				A\times\{(h_1,\dots,h_n)\}
				\,\setminus\, \mathcal{H}
				\right)
			}
		\subseteq \DL_A.\]
	\end{cor}
	
	\begin{proof}
\JIR{
To prove the first inclusion, we first note that $0\in\mathcal{N}_xX$ for any $x\in X$.
For any $a\in A$, there exists a sequence of points $a_k$ in $A_{\reg}$ converging to $a$.
Moreover, for each $k$, the pair $(a_k,0)$ is in $\Conormal(A)$. 
Because $\Conormal(A)$ is closed the limit $(a,0)$ is in $\Conormal(A)$.
Now for any point $a\in X$, repeat the argument above to find $(a,0)\in \Conormal(X)$. 
Hence for $a\in A\cap X$ it follows $(a,0)\in \Conormal(A)\cap \Conormal(X)$.
}
By Theorem~\ref{Structurethm} we get that  
	\[
	\overline{\Gamma((A\times\{0\})\setminus\mathcal{H})}\subseteq \DL_A.
	\]
		For the second inequality observe that 
		\[\pi_1(\Conormal(X)\cap \Conormal(A))\subseteq\pi_1(\Conormal(X))\cap \pi_1(\Conormal(A))=X\cap A,\] 
		and 
		\[\pi_2(\Conormal(X)\cap \Conormal(A))\subseteq\pi_2(\Conormal(X))\cap \pi_2(\Conormal(A))=Y_A\cap Y.
		\]
		We also have that $\Conormal(X)\cap \Conormal(A)$ is a subset of
		\[ \pi_1(\Conormal(X)\cap \Conormal(A))\times \pi_2(\Conormal(X)\cap \Conormal(A)).
		\]
		So we get that
		\[
		\Conormal(X)\cap \Conormal(A)\subseteq A\times (Y_A\cap Y).
		\]
		Hence
		\[
		\DL_A=\overline{\Gamma(\Conormal(X)\cap\Conormal(A)\setminus\mathcal{H})}\subseteq \overline{\Gamma((A\times Y_A\cap Y)\setminus\mathcal{H})}.
		\]
		The moreover part follows from the fact that $(h_1,\ldots,h_n)\in\mathcal{N}_aX $ and hence $(h_1,\ldots,h_n)\in\mathcal{N}_a A $ for any $a\in A$. By this we get that $(a,(h_1,\ldots,h_n))$ is an element of  $\Conormal(X)\cap \Conormal(A)$, so by Theorem~\ref{Structurethm} 
		we get that 
		\[\overline{
			\Gamma\left(
				A\times\{(h_1,\dots,h_n)\}
				\,\setminus\, \mathcal{H}
				\right)
			}
		\subseteq \DL_A.
		\]
	\end{proof}
	
	\begin{remark}\label{MLDSL}
		A specialized version of Corollary~\ref{Bound} for $\Gamma(x,y)=x+y$ (that is the Euclidean distance case) and for $A=\mathrm{Sing} X$ can be read as
		\[
		\mathrm{Sing}X\subseteq \DL_{\mathrm{Sing}X} \subseteq \mathrm{Sing}X+Y\cap Y_{\mathrm{Sing}X},
		\] where $Y$ is $\pi_2(\Conormal(X))$, 
		 and $Y_{\mathrm{Sing}X}$ is $\pi_2(\Conormal(\mathrm{Sing}X))$, 
		 and $``+"$ denotes the Minkowski sum of sets. 
		  While the first inclusion is trivial, the second inclusion is a stronger version of the second inclusion in ~\cite[Theorem~1]{EDDSL}, because we do not have the assumption of $X$ being a cone.
	\end{remark}
	
	\begin{remark}
		If $X$ is an affine cone that is an algebraic statistical model in $\mathbb{C}^{n+1}$, then $X$ is always a subset of the hyperplane controlling the sum of coordinates, that is $x_1+\ldots+x_n-x_{n+1}=0$ with $x_{n+1}=1$. 
		The
		algebraic approach to do maximum likelihood estimation on this model is by determining every critical point of the likelihood function on the model's closure. The \demph{(ML) maximum likelihood degree} of $X$ is said to be the number of critical points for generic data. 
		For this parametric optimization problem the corresponding choice of $\Gamma$ is \[
\Gamma(x,y)=x\star y:= (x_1 y_1,\ldots,x_n y_n,x_{n+1} y_{n+1})
\] 
with domain $(\mathbb{C}^{n+1}\times\mathbb{C}^{n+1})\setminus \mathcal{H}$, 
		where $\mathcal H \subset{\CC^{n+1}\times \CC^{n+1}}$ is defined by the equation $x_1\cdot x_2 \cdot\ldots \cdot x_{n+1} = 0$,
		\JIR{and $``\star"$ is called the Hadamard product}.
		
		Let $A$ be the singular locus of $X$. The corresponding data locus is called the \demph{ML Data Singular Locus} (see~\cite{MLDSL}) . 
		Observe that $\mathrm{Sing}X$ (as the whole model) is always a subset of the hyperplane controlling the sum of coordinates, 
		so in this case (the moreover part of ) Corollary~\ref{Bound} reads as
		\[
		(\mathrm{Sing}X\setminus \mathcal{H})\star( 1,1,\ldots,1,-1)\subseteq \DL_{\mathrm{Sing} X}\subseteq (\mathrm{Sing}X\setminus \mathcal{H})\star(Y\cap Y_{\mathrm{Sing}X}),
		\] where \JIR{in this case} $Y$ is the dual variety of $X$ \JIR{because $X$ is an affine cone},  
		$Y_{\mathrm{Sing}X}$ is the dual variety of $\mathrm{Sing}X$\JIR{for similar reasons},
		and $\star$ denotes the Hadamard product.
		A weaker version of this result appears in \cite[Theorem~$1$]{MLDSL}.
	\end{remark}
	\begin{example}[Rank $r$ approximations]\label{rank_var}
		An illustrating example of our main theorem is described next. Let $M^{\leq r}_{n}$ be the variety of $n\times n$ matrices of rank at most $r$. Let $q<r$ and let $A$ be the variety of $n\times n$ matrices of rank at most $q$. Now if we take $\Gamma(x,y)=x+y$, which is exactly the case of Euclidean distance optimization (see \cite[Example~$2.3$]{DHOST16}), then $\DL_{M^{\leq q}_n}$ corresponds to the set of all matrices that have at least one critical rank $r$ approximation that is in fact of rank at most $q$. For $q=r-1$, by \cite[Proposition $9$]{EDDSL}, we know that this set is equal to $M^{\leq n-1}_n$. Now for a general $q$ our structure theorem says that
		\[
		\DL_{M^{\leq q}_n}=\overline{
			\Gamma(\Conormal(M^{\leq r}_n)\cap\Conormal(M^{\leq q}_n)).
		}
		\]
		Now by \cite[Chapter $1$, Prop. $4.11$ and Lemma $4.12$]{GKZ} (and by \cite[Example 5.15]{RS13} for the symmetric case) we know that the conormal variety of  $M_n^{\leq r}$ is the set of pairs
		\[
		\{(A,B)\in M_n\times M_n,\text{ s.t. }A\in M_n^{\leq r}\text{ and }B\in M_n^{\leq n- r}\},
		\]
		and the conormal variety of  $M_n^{\leq q}$ is the set of pairs
		\[
		\{(A,B)\in M_n\times M_n,\text{ s.t. }A\in M_n^{\leq q}\text{ and }B\in M_n^{\leq n- q}\}.
		\]
		So we get that $\Conormal(M^{\leq r}_n)\cap\Conormal(M^{\leq q}_n)$ is equal to
\[
		\{(A,B)\in M_n\times M_n,\text{ s.t. }A\in M_n^{\leq q}\text{ and }B\in M_n^{\leq n- r}\}.
		\]
		Hence $\DL_A$ \JIR{is equal to}
		\[
		\Gamma(\Conormal(M^{\leq r}_n)\cap\Conormal(M^{\leq q}_n))=\Gamma(M_n^{\leq q} \times M_n^{\leq n-r})=M_n^{\leq n-r+q}.
		\]
		In in other words, we have that any matrix of rank at most $n-r+q$ will have at least one matrix of rank at most $q$ among its critical rank $r$ approximations.
		Choosing $q=r-1$ we get that any matrix with zero determinant  will have at least one matrix of rank at most $r-1$ among its critical rank $r$ approximations.
	\end{example}
	\begin{remark}
		Another consequence of Theorem~\ref{Structurethm} is a method on how to find points on $\DL_A$ computationally. First pick a \JIR{generic} point $a\in A$ and pick a set of generators $f_1,\ldots,f_c$ of the ideal of $X$. 
		\JIR{(Note it is not always easy to find such a point especially when using symbolic methods.)}
		 Then compute the Jacobian of $X$ at $a$, that is \[\Jac_x X|_a=\left(\frac{\partial f_i}{\partial x_j}\big|_a\right)_{i,j}.\] 
		Now any vector $v$ in the row span of $\Jac_x X|_a$ will be an element of $\mathcal{N}_a X\cap \mathcal{N}_a A$, hence the pair $(a,v)$ is an element of $\Conormal(X) \cap \Conormal(A)$. Finally by Theorem~\ref{Structurethm} we get that $\Gamma(a,v)$ is an element of $\DL_A$.
	\end{remark}
	\section{Computational examples}\label{Sec_Exa}
	Usually obtaining any knowledge on the conormal variety is very hard. When this is the case determining the data locus computationally helps a lot. Sometimes simply checking if the data locus is or is not the entire space proves certain general statements (see Example~\ref{Hankel_exa}). In what follows we will see a set of examples on how to compute these data loci.
	%The wide range of examples shows the importance as~well.
	\begin{example}[Rational normal surface]\label{rat_norm}
		Let us consider the rational normal surface in $\mathbb{C}^4$ (the twisted cubic surface). It is the image of the monomial map defined by
		\[
		(t_1,t_2)\mapsto (t_1^3, t_1^2t_2,t_1t_2^2, t_2^3).
		\] 
		The closure of the image is not a complete intersection, but it is defined by the vanishing of the three polynomials
		\[
		x_3^2-x_2x_4,\ x_2x_3-x_1x_4,\ x_2^2-x_1x_3.
		\]
		The origin is the only singular point of this variety and its intersection with a generic affine hyperplane is the \textit{moment curve}. We optimize a weighted distance~function
		\[
		\sum_{1\leq i\leq 4} w_i (u_i-x_i)^2,
		\]
		and we are interested in those points in $\mathbb{C}^4$ for which among the critical points of the weighted distance function to the rational normal surface there is at least one point on the moment curve cut out by $x_4=1.$ 
		So we have that $X$ is the rational normal surface and the subvariety $A$ is its cut by the hyperplane $x_4=1$. For this problem there are two natural weights to be considered for the distance function (see \cite[Example $2.7$]{DHOST16}). One is the unit weight $w_i=1$, corresponding to the classical Euclidean distance and the other one is the weight $w_i={3 \choose i-1}$. 
		\JIR{The latter corresponds to the natural metric on the space of symmetric tensors because it is the orthogonally invariant metric in the space $\Sym^3\CC^2$, namely it is the one induced by the action of $\SO(2,\CC)$ onto $\Sym^3\CC^2$.}
		We will choose the weight $w_i={3 \choose i-1}$ and we get that $\Gamma(x,y)=(x_1+\frac{1}{w_1}y_1,\dots, x_n+\frac{1}{w_n}y_n)$.
		After running the computations (see Example~\ref{comput_ex}) we get that the data locus is an irreducible hypersurface of degree~$7$. 
		The real part of an affine slice can be seen in Figure~\ref{RatNorm}.
		\begin{figure}[h]
			\begin{center}
				\vskip -0.3cm
				\includegraphics[scale=0.3]{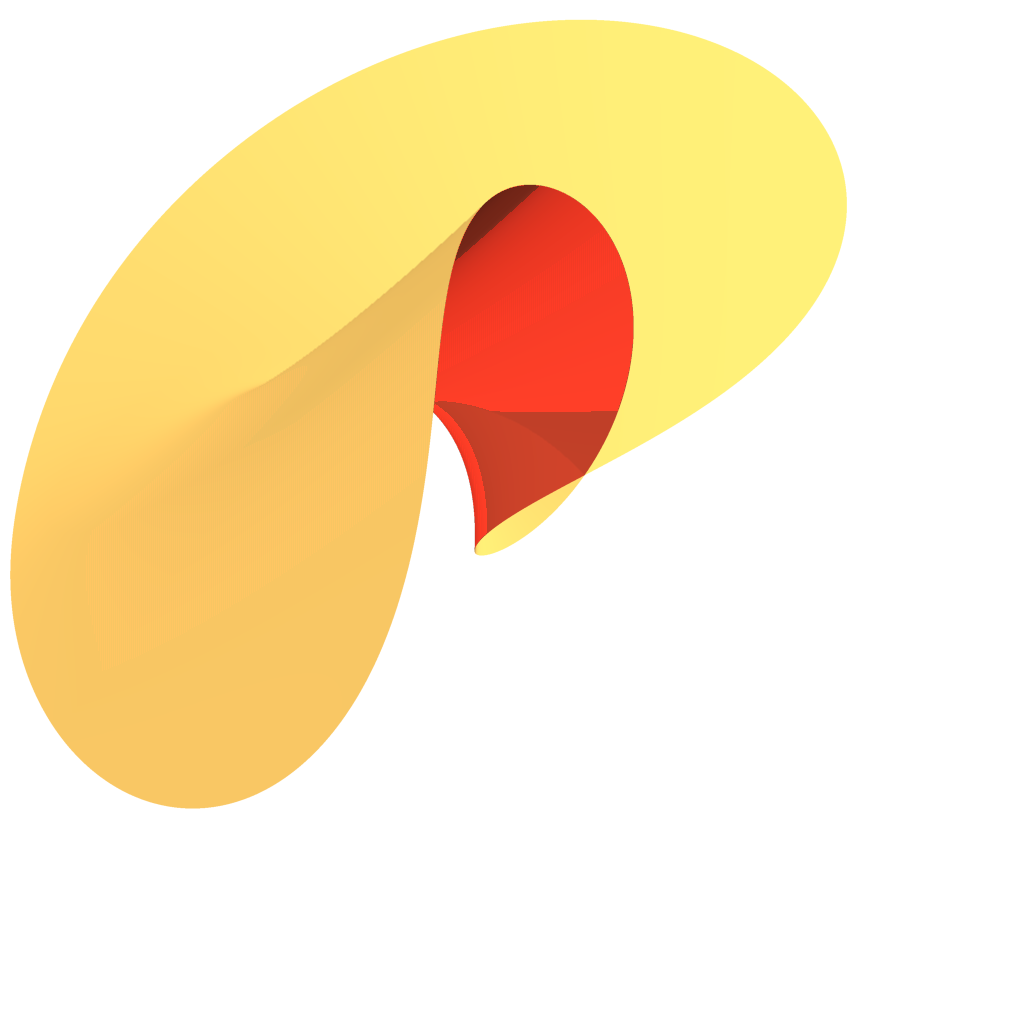}
				\vskip -0.4cm
				\caption{Data locus of the moment curve on the twisted cubic surface}
				\label{RatNorm}
			\end{center}
		\end{figure}
	\end{example}
	\begin{example}\label{comput_ex}
		Below is the {\tt Macaulay2} \cite{M2} code for computing the data locus of the moment curve from Example~\ref{rat_norm}.
		% \begin{leftbar}
		% \verbatiminput{examples/eddegree.txt}\vspace{-15pt}
		% \vspace{14pt}
		% \end{leftbar}
		\begin{leftbar}
			\begin{verbatim}
			n=4;
			kk=QQ[x_1..x_n,u_1..u_n];
			------------------------------------------------------------
			--defining polynomials of X
			(f1,f2,f3) = (x_3^2-x_2*x_4, x_2*x_3-x_1*x_4, x_2^2-x_1*x_3)
			X = ideal(f1,f2,f3);
			c = codim X;
			Jac = jacobian gens X;
			SingX = X+minors(c,Jac);
			------------------------------------------------------------
			g = x_1-1;--additional defining polynomial of the subvariety A
			A = X+ideal(g);
			------------------------------------------------------------
			Y = matrix{{u_1-x_1,3*(u_2-x_2),3*(u_3-x_3),u_4-x_4}};
			--Gamma is incorporated here, by setting y_i=w_i(u_i-x_i)
			S = submatrix(Jac,{0..n-1},{0..numgens(X)-1});
			Jbar = S|transpose(Y);
			projGammaCorr = X + minors(c+1,Jbar);
			--the (x,u) projection of the Gamma Correspondence
			projGammaCorrRegular=saturate(projGammaCorr,SingX);
			------------------------------------------------------------
			PreimDL=projGammaCorrRegular+A;--preimage of the data locus
			DLA = eliminate(toList(x_1..x_n),PreimDL);--data locus of A
			\end{verbatim}
		\end{leftbar}
		Here in the construction of the $\Gamma$-correspondence we use that 
		for $x\in X_{\reg}$, 
		$(w_i(u_i-x_i))_i\in \mathcal{N}_xX$ is equivalent to the matrix \[\begin{pmatrix}
		(w_i(u_i-x_i))_i\\
		\nabla f_1\\
		\vdots\\
		\nabla f_c
		\end{pmatrix}\] 
		having rank less than or equal to the codimension of $X$ (see \cite[Section $2$]{DHOST16}).
	\end{example}
	\begin{example}[Formation control~\cite{AndHelm}]\label{form_control}
		Formation control, for a set of $n$ points (called ``agents") in some given dimension $d$, is concerned with
		defining control laws \JIR{that} ensure that the points will move so that certain inter-agent distances approximate prescribed values as closely as possible.
		One of the challenging questions
		in this area is the following: given the dimension and the number of agents, what is the number of critical formations? This problem can be
		formulated as a distance optimization problem, see \cite[Example $3.7$]{DHOST16}. Here the authors
		proved a formula for the number of critical formations on a line ($d = 1$) for
		any number of agents. Using the notation from the above mentioned article let $X$ denote the variety in $\mathbb{C}^{\binom{p}{2}}$
		with parametric representation
		\begin{equation}
		\label{eq:CMpara}
		d_{ij}\,\, = \,\,(z_i-z_j)^2 \quad \hbox{for} \quad 1 \leq i < j \leq p.
		\end{equation}
		Thus, the points in $X$ record the squared distances among $p$
		interacting agents with coordinates  $z_1,z_2,\ldots,z_p \in \mathbb{C}^d$.
		Then $X$ is defined by the $(d+1) \times (d+1)$-minors of the
		{\em Cayley-Menger matrix}
		\begin{small}
			\[
			\begin{bmatrix}
			2 d_{1p} & d_{1p} {+} d_{2p} {-} d_{12} &  d_{1p} {+} d_{3p} {-} d_{13} &
			\cdots &  d_{1p} {+} d_{p-1,p} {-} d_{1,p-1} \\
			d_{1p} {+} d_{2p} {-} d_{12} & 2 d_{2p} &  d_{2p} {+} d_{3p} {-} d_{23} &
			\cdots &  d_{2p} {+} d_{p-1,p} {-} d_{2,p-1} \\
			d_{1p} {+} d_{3p} {-} d_{13} &  d_{2p} {+} d_{3p} {-} d_{23} & 2 d_{3p} &
			\cdots &  d_{3p} {+} d_{p-1,p} {-} d_{3,p-1} \\
			\vdots & \vdots & \vdots & \ddots & \vdots \\
			d_{1p} {+} d_{p-1,p} {-} d_{1,p-1} \! & \!
			d_{2p} {+} d_{p-1,p} {-} d_{2,p-1} \! & \!
			d_{3p} {+} d_{p-1,p} {-} d_{3,p-1} \! &  \!
			\cdots &    2 d_{p-1,p}
			\end{bmatrix}
			\]
		\end{small}
		Now we are interested in those tuples of prescribed inter-agent distances for which a special critical formation occurs. Finding these data is equivalent to determining a certain data locus.
		The first interesting case is for four agents in the plane, i.e. $p=4$ and $d=2$. We would like to find those data (or prescribed inter-agent distance tuples) for \JIR{which a critical formation is a square.}
		Hence $A$ is the subvariety satisfying the additional constrains
		\[
		d_{12}-d_{23}=d_{23}-d_{34}=d_{34}-d_{14}=0,
		\]
		and we are interested in $\DL_A$. After running the computations we get that the locus of these data forms a degree $13$, codimension $3$ variety generated minimally by $26$ polynomials. The {\tt Macaulay2} code and the resulting polynomials of the computations can be found at~\cite{WebStorage}.
%		\JIR{To your knowledge these findings regarding $\DL_A$ for this choice of $A$ are new.}
	\end{example}
	\begin{example}[Structured low rank approximations~\cite{OttSpSt}]\label{Hankel_exa}
		Structured low-rank approximation is the problem of minimizing the Frobenius distance to a given matrix among all matrices of fixed rank in a linear space of matrices.
		Here we optimize a  distance function
		\[
		\sum_{1\leq i,j\leq n} (u_{ij}-x_{ij})^2.
		\]
		%In this case we have again that $\Gamma(x,y)=(x_{ij}+\frac{1}{w_{ij}} y_{ij})_{i,j}$.
		In this case we have again that $\Gamma(x,y)=(x_{ij}+y_{ij})_{i,j}$.
		The {\em Hankel matrix} $H_n$ of format $n \times n$ has the entry $x_{i+j-1}$ in row $i$ and column $j$.
		For example
		\begin{equation}
		\label{ex:hankel56}
		H_5  = \begin{bmatrix}
		x_1 & x_2 & x_3 \\
		x_2 & x_3 & x_4 \\
		x_3 & x_4 & x_5
		\end{bmatrix}.
		\end{equation}
		Any element of the  space of symmetric $2 {\times} 2 {\times} \cdots {\times} 2$-tensors corresponds to a binary form
		$$ F(s,t) \,= \,
		\sum_{i=1}^n \binom{n-1}{i-1} \cdot x_i \cdot s^{n-i} \cdot t^{i-1}. $$
		For such a binary form the corresponding Hankel matrix $H_n$ has rank $1$ if and only if $F(s,t)$ is the $(n{-}1)$-st power of a linear form. More generally,
		if $F(s,t)$ is the sum of $r$ powers of linear forms, then $H_n$ has rank $\leq r$. This locus corresponds to the $r$-th secant variety of the rational normal curve (see Example~\ref{rat_norm} \EH{and for more details on this topic consult \cite[Section $4.3.7$]{Land}}).
		
%		For low rank approximations of Hankel matrices there are certain weights to be considered (see more \cite[Section $4$]{OttSpSt}). 
		%\JIR{We choose to work with the unit weight vector $w_{ij}=1$, because of computational reasons.}
		%\JIR{Emil wants to update this example to address the referee's point.}\EH{...and maybe Jose wants to run some additional computations.}
		Based on the fact that  the best rank $1$ approximation to real symmetric tensors can be chosen symmetric (see~\cite{Fried}), we are interested in the following problem. What is the set of those $3\times 3$ Hankel matrices with Hankel matrix among their critical rank $1$ approximations? Can the best rank $1$ approximation be chosen to be Hankel?
		
		For some structured low rank approximations the answer to the analogous question is true. For example a symmetric matrix always has a symmetric matrix among its rank $1$ critical approximations (so for $2\times 2$ Hankel matrices this is true).

		An analogy is not true for Hankel matrices of dimension $n\geq 3$ \JIR{by our computations.}
		Choosing the variety $X$ to be $3\times 3$ matrices of rank one and letting the subvariety $A$ to be the rank one Hankel matrices, the data locus $\DL_A$ consists of all matrices that admit at least one critical Hankel rank one approximation. 
		After running our code, we get that there is an irreducible hypersurface of degree $3$ of $3\times 3$ matrices which have a critical rank $1$ Hankel approximation. This hypersurface has a codimension $5$, degree $9$ subvariety of Hankel matrices.
		The {\tt Macaulay2} code and the resulting polynomials of the computations can be found at~\cite{WebStorage}.
	\end{example}
	\begin{example}[Low rank approximation of tensors]\label{low_rank}
		\JIR{In this example we will approximate a tensor with one of low rank. How well we approximate is measured by the Frobenius norm.}
		\JIR{Now we ask the question,} when does it happen that among the critical rank $2$ approximations of an $n_1\times n_2\times\ldots \times n_p$ tensor there is a tensor that is of rank $1$?
		
		The smallest interesting case would be for $2\times 2\times 3$ tensors. 
		\JIR{
		Informally, recall that the rank of a tensor is the smallest smallest $r$ such that it can decomposed as a sum of $r$ pure (rank one) tensors. 
		A tensor is said to have border rank $r$ if it can be expressed as a limit of tensors with rank $r$. When $r=1$ these two notions of rank agree. }
		\EH{For details see~\cite[Section $2.4$]{Land}.}

		Let $Rk_2$ be the variety of border rank at most $2$ tensors, defined by all the $3\times 3$ minors of all the flattenings. \EH{In general if, $V_1, V_2, V_3$ are inner product spaces and $\{1,2,3\}=J_1\cup J_2\}$ is a partition, then the map
		\[V_1\otimes V_2 \otimes V_3\mapsto (\bigotimes_{j\in J_1}V_j)\otimes (\bigotimes_{j\in J_2}V_j),\]
		that sends the set of order-$3$ tensors into the set of order-$2$ tensors (matrices), where the inner product on each factor $\bigotimes_{j\in J_k}V_j$ is the one induced from the inner products on the factors is called a flattening.}
		
		Now let $Rk_1$ be the subvariety of rank at most one tensors, defined by all the $2\times 2$ minors of all the flattenings, \EH{see for instance \cite[Section $7.3$]{Land}}. The singular locus of $Rk_2$ is defined by all the $2\times 2$ minors of all but one flattening. The missing minors come from a $2\times 6$ flattening of the tensor. So we have that $Rk_1\subset \mathrm{Sing}_{Rk_2}$.
		After running the computations we get the following. The data locus of the singular locus is $Rk_2$ itself. So $\DL_{Rk_2}=Rk_2$. This means that all tensors (and only them) of rank at most $2$ have a singular critical rank $2$ approximation. Moreover the data locus of rank one tensors is a subvariety of the previous data locus, that is \[\DL_{Rk_1}\subset \DL_{\mathrm{Sing}_{Rk_2}}=Rk_2.\] 
		Also $\DL_{Rk_1}$ has codimension $3$, degree $40$ and is defined by $10$ polynomials. The {\tt Macaulay2} code and the resulting polynomials of the computations can be found at~\cite{WebStorage}.
	\end{example}

\begin{example}[Kalman Varieties]\label{ex:kalman}
\JIR{
Consider the space of $n_1\times\cdots \times n_k$ tensors. 
The singular vector-tuples of a tensor $U$ is defined in \cite{Lim-Tensor}.
A variety of tensors with singular vector-tuples satisfying certain algebraic conditions is called a Kalman variety \cite{OttSha, OS13, ShaSoVen}.
Kalman varieties are an example of data loci because
the critical points of the Euclidean distance function from $U$ restricted to the variety of rank one tensors correspond to the singular vector-tuples of the tensor $U$ \cite[Equation 6]{Lim-Tensor} and \cite[Section 8]{DHOST16}.
}
\end{example}

	\noindent
	{\bf Acknowledgements.} E.~Horobe\c{t} is supported by Sapientia Foundation - Institute for Scientific Research, Romania, Project No. $17/11.06.2019.$
	We thank Bernd Sturmfels 
	for initiating this project and
	for his helpful comments. 
	We also thank Luca Sodomoco, Kaie Kubjas, 
	Olga Kuznetsova,
	and Elias Tsigaridas for their helpful feedback. 

	We are immensely  grateful for the referees wonderful suggestions which have improved the quality of the article.

	\vspace{1cm}
	\footnotesize {\bf Authors' addresses:}
	
	\smallskip
	
	\noindent Emil Horobe\c{t}, Sapientia Hungarian University of Transylvania \ 
	\hfill {\tt horobetemil@ms.sapientia.ro}\\
	 Targu-Mures - Corunca, nr. 1C. Postal code: 540485%, Op. 9., Cp. 4
	
	\noindent Jose Israel Rodriguez,
	University of Wisconsin-Madison	 \hfill {\tt jose@math.wisc.edu}\\
	Department of Mathematics, Van Vleck Hall,
	480 Lincoln Drive, Madison, WI 53706


\begin{thebibliography}{10}
		
		\bibitem{WaterFilling} E. Altman, K. Avrachenkov and A. Garnaev, \textit{Closed form solutions for water-filling problems in optimization and game frameworks}, Telecommunication Systems, $47(1)$, $(2011)$, $153-164$.
		
		\bibitem{AndHelm}B. Anderson and U. Helmke,\textit{ Counting critical formations on a line}, SIAM J. Control Optim. 52 ($2014$), pp. $219-242.$
		
		\bibitem{ConvOpt} Boyd S., Boyd S. P. and Vandenberghe, L., \textit{Convex optimization}, Cambridge university press, $2004.$

		\bibitem{DR07} L.~Daniel and F.~ Rouillier, \textit{Solving parametric polynomial systems}, Journal of Symbolic Computation 42.6 (2007): 636-667.
		
		\bibitem{DHOST16}
		J.~Draisma, E.~Horobe\c t, G.~Ottaviani, B.~Sturmfels and R.R.~Thomas:
		{\em The Euclidean distance degree of an algebraic variety},
		Foundations of Computational Mathematics $16 (2016)$, pp. $99-149.$
		
		\bibitem{Fried}S. Friedland, \textit{Best rank one approximation of real symmetric tensors can be chosen symmetric}, Frontiers of Mathematics in China $8.1 (2013)$, pp. $19-40.$
		
		\bibitem{GKZ} I.M. Gelfand, M.M. Kapranov, and A.V. Zelevinsky, \textit{Discriminants, Resultants and Multidimensional Determinants}, Birkh\"auser, Boston, 1994.
		
		\bibitem{M2} D.~Grayson and M.~Stillman:
		{\em Macaulay2, a software system for research in algebraic geometry}, available at
		{\tt www.math.uiuc.edu/Macaulay2/}.
		
		\bibitem{EDDSL}
		E. Horobe\c{t}, {\em The Data Singular and the Data Isotropic Loci for Affine Cones}, Comm. Algebra, Volume $45$ ($2017$), Issue $3$, pp. $1177 -1186$.
		
		\bibitem{MLDSL}
		E. Horobe\c{t}, J. I. Rodriguez, {\em The Maximum Likelihood Data Singular Locus}, J. Symbolic Comput., Volume $79$ ($2017$), Part $1$, pp. $99-107$.

		\bibitem{KKS21}K.~Kubjas, O.~ Kuznetsova, and  L.~Sodomaco, \textit{Algebraic degree of optimization over a variety with an application to $ p $-norm distance degree}, (2021), arXiv preprint arXiv:2105.07785.
		
		\bibitem{Land} J. M.~Landsberg, \textit{Tensors: geometry and applications}, Representation theory $381$, no. $402 (2012): 3$.

		\bibitem{Lass}J.B.~Lasserre, \textit{An introduction to polynomial and semi-algebraic optimization}, Vol. 52. Cambridge University Press, 2015.

		\bibitem{Lim-Tensor} L.-H. Lim, \textit{Singular values and eigenvalues of tensors: a variational approach}, Proc.
IEEE International Workshop on Computational Advances in Multi-Sensor Adaptive Processing (CAMSAP 05), 1 (2005), 129-132.
		
		\bibitem{MR4048615}
		L. Maxim, J. I. Rodriguez, B. Wang, {\em Euclidean distance degree of the multiview variety}, SIAM J. Appl. Algebra Geom., Volume $4 (2020)$, pp. $23-48$.

		\bibitem{OttSha} G.~Ottaviani and Z.~ Shahidi, \textit{Tensors with eigenvectors in a given subspace}, Rend. Circ. Mat.
Palermo, II. Ser, pages 1–12, 2021. 

		
		\bibitem{OttSpSt} G. Ottaviani, P.-J. Spaenlehauer and B. Sturmfels,\textit{Exact solutions in structured low-rank approximation}, SIAM Journal on Matrix Analysis and Applications $35.4 (2014)$, pp. $1521-1542.$

		\bibitem{OS13} G. Ottaviani and B. Sturmfels, \textit{ Matrices with eigenvectors in a given subspace}, Proc. Amer. Math. Soc.,
141(4):1219–1232, 2013
		
		\bibitem{RS13} P. Rostalski and B. Sturmfels, \textit{Dualities}, Chapter 5 in G. Blekherman, P. Parrilo and
		R. Thomas: \textit{Semidefinite Optimization and Convex Algebraic Geometry}, pp. $203-250,$
		MPS-SIAM Series on Optimization, SIAM, Philadelphia, $2013.$
		
		\bibitem{ShaSoVen} Z. Shahidi, L. Sodomaco and E. Ventura,\textit{ Degrees of Kalman varieties of tensors}, (2021), arXiv preprint arXiv:2109.09481.
		
		\bibitem{ThJoDo}J. B. Thomassen, P.H. Johansen and T. Dokken,\textit{Closest points, moving surfaces, and algebraic geometry}, Mathematical methods for curves and surfaces: Troms\o, ($2004$), pp. $351-362.$
		
		\bibitem{WebStorage} \url{https://github.com/JoseMath/Data-Loci-Supplementary-Material}
		
	\end{thebibliography}
\end{document}